\documentclass[12pt,a4paper]{article}
\usepackage{euscript,amsfonts,amssymb,amsmath,amscd,amsthm}

\usepackage{color}
\usepackage[T2A]{fontenc}

\newtheorem{theorem}{Theorem}[section]
\newtheorem*{theorem*}{Theorem}
\newtheorem{lemma}[theorem]{Lemma}

\newtheorem{proposition}[theorem]{Proposition}

\newtheorem{corollary}[theorem]{Corollary}

\newcommand{\Sym}{{\mathfrak S}}

\renewcommand{\O}{{\mathcal O}}

\newcommand{\R}{{\mathcal R}}

\begin{document}

\title{Record-dependent measures \\on the symmetric groups}

\author{Alexander Gnedin\thanks{
School of Mathematical Sciences, Queen Mary University of London, Mile End Road, London E1 4NS,
UK; e-mail: a.gnedin@qmul.ac.uk}
 \and Vadim Gorin\thanks{Massachusetts Institute of Technology, 77 Massachusetts avenue, Cambridge, MA 02140, USA;
 Institute for Information Transmission Problems, Bolshoy Karetny 19, Moscow 127994, Russia;
 e-mail: vadicgor@gmail.com} }
\date{25.09.13}

\maketitle

\begin{abstract}
\noindent A probability measure  $P_n$ on the symmetric group ${\mathfrak S}_n$ is said to be
record-dependent if $P_n(\sigma)$  depends only on the set of records of a permutation $\sigma\in
{\mathfrak S}_n$. A sequence $P=(P_n)_{n\in{\mathbb N}}$ of consistent record-dependent measures
determines a random order on $\mathbb N$. In this paper we  describe the extreme elements of the
convex set of such $P$. This problem turns out to be related to the study of asymptotic behavior
of permutation-valued growth processes, to random extensions of partial orders, and to the
measures on the Young-Fibonacci lattice.

\end{abstract}


\section{Introduction}
\label{Section_Intro} Let $\Sym_n$ be  the group of permutations of  $[n]:=\{1,\dots,n\}$. A
position $j\in [n]$ is called an {\it upper record position} or simply a {\it record} in a
permutation $\sigma\in\Sym_n$ if $\sigma(j)=\max_{i\in [j]}\sigma(i)$. Let $R(\sigma)\subset [n]$
be the set of records of $\sigma$. A probability measure $P_n$ on ${\mathfrak S}_n$ is called {\it
record-dependent} (RD) if
 $P_n$ is conditionally uniform given the set of records,
or, equivalently, if the probability mass function $P_n(\sigma)$  depends only on $R(\sigma)$.

A natural way to connect permutations of different sizes is suggested by  viewing a generic
permutation $\sigma\in\Sym_n$  as a (total) order on $[n]$, in which $i$ precedes $j$ if $i$
appears in a lower position, that is $\sigma^{-1}(i)<\sigma^{-1}(j)$. Restricting the order to the
smaller set $[n-1]$ yields a projection $\pi_{n-1}^n:\Sym_n\to \Sym_{n-1}$, by which each permutation
$\sigma$ is mapped to a permutation which we call {\it coherent} with $\sigma$. Likewise, two
probability measures, $P_n$ on $\Sym_n$ and $P_{n-1}$ on $\Sym_{n-1}$, are said to be coherent if the
restriction sends $P_n$ to $P_{n-1}$. It turns out  that for coherent measures if $P_n$ is  RD,
so is $P_{n-1}$.

In this paper we are interested in  coherent sequences of RD-measures $P=(P_n)_{n\in{\mathbb N}}$.
Each such sequence defines a Markovian permutation growth process, with backward transition
probabilities like under the uniform distributions. On the other hand, a coherent sequence
uniquely determines a probability measure, denoted by the same symbol $P$, on a `$n=\infty$'
object, which is the space of orders on ${\mathbb N}$. The measure $P$ will be called an
RD-measure, meaning that the projection of $P$ to each $[n]$ is RD.
Our main results (Theorem \ref{theorem_extreme_points} and Proposition
\ref{prop_law_of_large_numbers}) explicitly characterize the extreme elements of the  convex set
of such RD-measures $P$. The characterization problem belongs to the circle of de Finetti-type
questions around sufficiency and stochastic symmetries \cite{Aldous, Olav}, and can be viewed in
 different contexts like the boundary problem for a branching scheme
\cite{Kerov_book} or processes on causal posets \cite{BL1,BL2}.

A straightforward example  of an RD-measure on $\Sym_n$ is
\begin{equation}\label{Ewens}
P_n(\sigma)=\frac{1}{Z}\, \theta^{|R(\sigma)|}, ~~~\sigma\in\Sym_n,
\end{equation}
where $\theta\in (0,\infty)$ and $Z=\theta(\theta+1)\cdots(\theta+n-1)$. For $\theta=1$ this is
the uniform distribution. By the `fundamental bijection',  which identifies the one-row notation
for writing a permutation with the cycle notation for another permutation, this measure is mapped
to the well-known Ewens distribution with weights $\frac{1}{Z}\theta^{|{\rm cycles}(\sigma)|}$
\cite{KerovTsi, DM}. However, unless $\theta=1$
the measures (\ref{Ewens}) are not coherent in the sense of the present paper. A qualitative
difference appears if we look how the number of records $|R(\sigma)|$ grows with $n$. For
(\ref{Ewens}) the order of growth is logarithmic, while for the measures studied here the right
scale  for $|R(\sigma)|$ is linear, with the case of the uniform distribution being the sole
exception.

The case of the uniform distribution is special among the RD measures. The uniform distribution on
$\Sym_n$ is important in statistics, as it appears by ranking samples from continuous
distributions. In this  connection various quantities, such as record times, record values,
interrecord times and others attracted lots of attention. We refer the reader to the review
\cite{Records_review} for  classical results on the theory of records.

In   the literature there is a number of other  permutation growth models build on  principles
similar to ours: the probability mass function depends, for each $n$, on a statistic $S$, while
coherence of random permutations of different sizes is defined via a system of projections. When
$S$ is the cycle type of a permutation,  the model can be embedded into Kingman's theory of
exchangeable partitions \cite{Kerov_book, CSP}.
When $S$ is the set of descents,  a coherent  sequence of random permutations  is associated with
a random order on $\mathbb N$ which has the property of spreadability \cite{GO}. Also Pitman's
partially exchangeable partitions \cite{PitmanPTRF} and random sequences of compositions of
integers \cite{KerovSubord, DM} can be recast in terms of  coherent random permutations.
Furthermore, certain parametric deformations of the uniform distribution fit in the framework,
with $S$ being a numerical statistic like the number of cycles, descents, pikes, inversions, etc
(see \cite{Mallows, DM} for examples and references).

Both the choice of a statistic $S$ and the choice of a system of projections connecting symmetric
groups affect properties of the permutation growth model. If we replace $|R(\sigma)|$ in
(\ref{Ewens}) by the number of {\it lower} records, the measures will become coherent under our
projections $\pi_{n-1}^n$. On the other hand, changing the system of projections one can achieve
the coherence of distributions (\ref{Ewens}). The alternative projections were considered in
\cite{DM} for  permutations with distribution depending on statistics of  both upper and lower
records.

 Comparing with the previous work, the main distinction of the present setting is in the
structure of the   set of extreme RD-measures $P$. In e.g.\ \cite{ DM, Kerov_book, KerovSubord,
PitmanPTRF} the extremes are described in terms of infinitely many continuous parameters, which
are asymptotic frequencies (on a linear scale) of certain components of a growing permutation. In
contrast to that, in our model the parametrization of the extreme RD-measures involves an integer
sequence and a real number.

Somewhat unexpectedly, a parametrization of extremes similar to ours has appeared  in the work of
Goodman and Kerov \cite{YF1}.
They studied random growth  processes on the Young-Fibonacci lattice, which
is an important example of a differential poset (as introduced
 by Stanley \cite{YF_St} and Fomin \cite{YF_F}), and proved a result  very much in line with our Theorem \ref{theorem_extreme_points}.
Although we see some further similarities with the setting of \cite{YF1},
the  connection remains obscure,
and it would be very interesting to  have more clarity in this point.

A prototypical instance of  differential poset is  Young's lattice. The study of   {coherent}
measures on Young's lattice, that is  growth processes with values in Young diagrams,  is a deeply
explored subject related, in particular, to the theory of group representations, total positivity
of matrices, and asymptotics of Schur symmetric functions,
see
\cite{Kerov_book, VK_Long} and references therein. In particular, the identification of the extreme
coherent measures
is equivalent to the classification of the characters and finite factor representations of the
infinite symmetric group ${\mathfrak S}_\infty$. Our results on the RD-measures can be also
interpreted  in this  spirit, as the classification of  traces on the $AF$--algebra associated
with the branching scheme (Bratteli diagram) of  permutations.

Finally, we  mention that some of the extreme RD-measures can be viewed in the context of
order-invariant measures on fixed causal sets, as introduced recently by Brightwell and Luczak
\cite{BL2}. We add details to this  aspect of our study in Section \ref{Section_connections}.

\section{Record-dependent measures and orders on $\mathbb N$}

We write permutations $\sigma\in\Sym_n$ in the one-row notation as permutation words
$\sigma(1)\dots\sigma(n)$. Note that position 1 is the smallest and $\sigma^{-1}(n)$ is  the
largest  element of the set of records $R(\sigma)$. For instance,
permutation $\sigma=2\,6\,5\,7\,1\,4\in\Sym_7$ has records $R(\sigma)=\{1,2,4\}$.

With $\sigma\in \Sym_n$ one associates an order on $[n]$, in which letter $i$ precedes $j$ if $i$
appears in a lower position, meaning that $\sigma^{-1}(i)<\sigma^{-1}(j)$. The restriction of the
order to $[n-1]$ yields a  projection $\pi^n_{n-1}:\Sym_n\to \Sym_{n-1}$,
which amounts to removing  letter $n$
from the permutation word. Explicitly,  for $i\in [n-1]$
$$
 \pi_{n-1}^n(\sigma)(i)=\begin{cases} \sigma(i), &{\text{if }} i<\sigma^{-1}(n),\\ \sigma(i+1),& {\text{if }}
 i\ge\sigma^{-1}(n).
                   \end{cases}
$$
For example, $\pi_3^4$ sends $3\,4\,1\,2$ to $3\,1\,2$. More generally,
the iterated  projection $\pi_m^n:
\Sym_n\to\Sym_m$ is defined for $1\leq m<n$ as the operation of deleting letters $m+1,\dots,n$
from the permutation word $\sigma(1)\dots\sigma(n)$.


Let $\O$ be the projective limit of the symmetric groups $\Sym_n$ taken together with projections
$\pi_{n-1}^n$. Thus, an element of $\O$  is a  coherent sequence $(\sigma_n)_{n\in{\mathbb N }}$, which
has  $\sigma_n\in\Sym_n$ and  $\sigma_{n-1}=\pi_{n-1}^n(\sigma_n)$   for $n>1$.
 Let $\pi^{\infty}_n$ denote the coordinate map sending
$(\sigma_1,\sigma_2,\dots)\in \O$ to $\sigma_n\in\Sym_n$. The coherence of permutations
immediately implies the following statement.

\begin{proposition}
 $\O$ is in bijection with the set of total orders on  ${\mathbb N}$. The projection
$\pi_n^\infty$ amounts to restricting the order from $\mathbb N$ to $[n]$.
\end{proposition}

\noindent
In what follows we identify elements of $\O$ with the orders corresponding
to them. We endow each $\Sym_n$ with the discrete topology, and endow $\O$ with the
product topology of projective limit, which corresponds to the  coordinate-wise convergence.
In this topology $\O$ is a compact totally disconnected space.

Let $\mathcal M(\O)$ be the space of Borel probability measures on $\O$. Each measure $P\in
\mathcal\O$ has marginal measures $P_n:=\hat{\pi}^\infty_{n}(P)$ which  satisfy the coherence
condition $P_{n-1}=\hat{\pi}^n_{n-1}(P_n)$ for $n>1$, where and henceforth $\hat{f}$ denotes the
pushforward of  measures under a mapping $f$. Conversely, by Kolmogorov's measure extension
theorem each coherent sequence $(P_n)$ determines a unique measure on $\O$.

Let $\mathcal M_R(\Sym_n)$ be the set of RD-measures on $\Sym_n$, as defined in the Introduction.
We call $P\in \mathcal M(\O)$ an RD-measure if  the marginal measures satisfy $P_n\in \mathcal
M_R(\Sym_n)$ for every $n$. We denote $\mathcal M_R(\O)$ the convex set of such RD-measures, and denote
${\rm ext} \,\mathcal M_R(\O)$ the set of extreme elements of $\mathcal M_R(\O)$.

The permutation statistic $R(\sigma)$ is consistent
with projections $\hat \pi_{n-1}^n$ in the following sense.
\begin{lemma}
\label{prop_stability_of_record_dependence}  If $P_n$ is an RD-measure  on $\Sym_n$
 then its projection $P_{n-1}:=\hat{\pi}^n_{n-1}(P_{n})$ is an RD-measure on $\Sym_{n-1}$.
\end{lemma}
\begin{proof} Recall that
for $\sigma \in \Sym_n$, the position $\sigma^{-1}(n)$ is the maximal element of $R(\sigma)$. Note
that $\sigma$ can be uniquely recovered from $\sigma^{-1}(n)$ and $\pi_{n-1}^n(\sigma)$. Thus, for
$\tau\in \Sym_{n-1}$ and $A\subset[n]$ the number $k$ of permutations $\sigma$ which satisfy
$R(\sigma)=A$ and $\tau=\pi_{n-1}^n(\sigma)$ depends only on $A$ and $B:=R(\tau)\subset[n-1]$.
Specifically, $k=1$ if for some $j\in[n]$ the set $A$  can be obtained by the
following deletion-insertion operation: delete from $B$ all elements greater than $j-1$ then
insert $j$ in the remaining set. Otherwise $k=0$. For instance, taking $n=7$ and
$B=\{1,3,5\}$  we have $k=1$ for
$$A=\{1\}, \{1,2\},  \{1,3\},   \{1,3,4\},  \{1,3,5\},  \{1,3,5,6\},  \{1,3,5,7\}.$$
The assertion follows since
$$P_{n-1}(\tau)=\sum_{A\subset [n]}~~\sum_{\sigma:\pi_{n-1}^n(\sigma)=\tau,R(\sigma)=A} P_n(\sigma)$$
depends only on $B$.
\end{proof}

An alternative  coordinatization of permutations is sometimes useful.
For a permutation $\sigma$
define $r_i$ as
the rank of $\sigma(i)$ among $\sigma(1),\dots,\sigma(i)$. That is to say, $r_i=k$ if $\sigma(i)$
is the $k$th smallest element in $\{\sigma(1),\dots,\sigma(i)\}$. We call the $r_i$'s {\it ranks}
 (other terminology found in the literature is `relative ranks' or `initial ranks').
The correspondence $\sigma\mapsto (r_1,\dots,r_n)$ is a bijection between $\Sym_n$ and
$[1]\times\cdots\times[n]$. Under the uniform distribution on $\Sym_n$ the ranks are independent random
variables, with $r_i$ uniformly distributed  on $[i]$.

\vskip0.2cm \noindent {\bf Remark.} The representation of permutations by rank sequences does not
sit well with the projections $\pi_{n-1}^n$, which in terms of the ranks $r_i$ are rather
involved.    In \cite{DM} other projections $(r_1,\dots,r_n)\mapsto (r_1,\dots,r_{n-1})$ were used
to study measures (\ref{Ewens}) and their generalizations.

\vskip0.2cm
For $\rho\subset[n]$ satisfying $1\in \rho$, let $P^\rho$ be the {\it elementary measure},
which is the uniform distribution on the set of permutations $\{\sigma\in {\frak S}_n: R(\sigma)=\rho\}$.
Note that  $P^\rho$ is a product measure in the rank coordinates: the $r_i$'s are independent,
 $r_i$ is uniformly distributed on $[i-1]$ for $i\notin \rho$, and $r_i= i$ almost surely for $i\in\rho$.
The set of RD-measures $\mathcal M_R(\Sym_n)$ is a simplex with $2^{n-1}$ extreme elements $P^\rho$.


The convex set $\mathcal M_R(\O)$ is a projective limit of the finite-dimensional simplices $\mathcal M_R(\Sym_n)$.
By the general theory (see e.g. \cite{Goodearl})  $\mathcal M_R(\O)$ is a
Choquet simplex, i.e. a convex compact set with the property of uniqueness of the representation
of a generic point as a convex mixture of the elements of  ${\rm ext} \,\mathcal M_R(\O)$.
 In view of this property it is important to determine the set of extreme RD-measures  ${\rm ext} \,\mathcal M_R(\O)$.

\section{Constructions of the extreme RD-measures}
\label{Section_Statement}

A natural concept of `a uniformly distributed random order' on $\mathbb N$ is the
probability  measure $P^*\in\mathcal M_R(\O)$  whose projection to $\Sym_n$ is the uniform distribution for every $n$.
A characteristic feature of  $P^*$ is  exchangeability, that is the invariance under bijections of $\mathbb N$.
The order can be neatly constructed in terms of a sequence $(\xi_i)$ of independent random
variables uniformly distributed on the unit interval, by letting  $i$ to precede $j$  iff
$\xi_i<\xi_j$. It is clear from this construction that the exchangeable order  is
almost surely dense and has neither maximal, nor minimal elements. Thus, $\mathbb N$ with this
order is isomorphic (as an ordered space) to $({\mathbb Q},<)$ $P^*$-almost surely, as it follows
from the classical characterization of dense orders  due to Hausdorff  (see \cite[Section III.11]{Haus}).

Under the uniform distribution on $\Sym_n$
the number of records
satisfies $|R(\sigma)|/\log n\to 1$ in probability, as is well known  \cite{Records_review}. It will be
clear from what follows that $P^*$ is the only RD-measure with sublinear growth of $|R(\sigma)|$
as $n\to\infty$. We introduce next  a family of random orders for which  the number of records is
asymptotically linear in $n$. The idea is to exploit  the ranks as in the construction of extreme
elements of $\mathcal M_R(\Sym_n)$. To that end, we need some preliminaries.

We define an infinite permutation as a bijection $\sigma:{\mathbb N}\to{\mathbb N}$
and  denote  $\Sym$ the set of  such bijections.  The symmetric group  ${\mathfrak S}_n$ is naturally embedded in $\mathfrak S$ as the set of bijections
 that satisfy $\sigma(j)=j$ for $j>n$, and
the infinite symmetric
group   is identified with ${\mathfrak S}_\infty=\cup_{n=1}^\infty {\mathfrak S}_n$.
Each $\sigma\in {\mathfrak S}$ defines an order on $\mathbb N$ by the familiar rule: $i$
precedes $j$ iff $\sigma^{-1}(i)<\sigma^{-1}(j)$. This order is of  the type of  the ordered set $({\mathbb N},<)$,
with the $j$th minimal element being $\sigma(j)$. By the virtue of this correspondence
$\Sym$ is embedded in $\O$. Note that the support of the exchangeable order is disjoint with
$\Sym$, i.e.\ $P^*(\Sym)=0$.

With $\sigma\in\Sym$ we associate an infinite sequence of ranks $r_1,r_2,\dots$, where $r_i$ is
the rank of $\sigma(i)$ among $\{\sigma(1),\dots,\sigma(i)\}$. More generally, for $i\leq j$ let
$r_{i,j}$ be the rank of $\sigma(i)$ among $\{\sigma(1),\dots,\sigma(j)\}$. The bivariate array
$r_{i,j}$ is determined by the diagonal entries $r_i$ by the virtue of the  recursion
\begin{equation}\label{rec-ranks}
r_{i,i}=r_i,~~~r_{i,j+1}=r_{i,j}+1(r_{j+1}\leq r_{i,j}),
\end{equation}
where $1('C')$ is 1 when the condition $'C'$ is true and is $0$ otherwise. Moreover, the
sequence $r_{i,i},r_{i,i+1},\dots$ is nondecreasing and eventually stabilizes at the value
\begin{equation}\label{sigma-r}
\sigma(i)=\lim_{j\to\infty} r_{i,j}.
\end{equation}
Thus, $\sigma\in \Sym$ is uniquely determined by the sequence of ranks $(r_i)$. This correspondence
suggests a criterion to identify the sequences of ranks corresponding to infinite permutations.
\begin{lemma}\label{inf-corr}
A sequence $(r_i)_{i\in {\mathbb N}}$ with $r_i\in[i]$ defines a bijection  $\sigma\in \Sym$ iff for every $i$ the
nondecreasing sequence $r_{i,i},r_{i,i+1},\dots$ defined recursively by {\rm (\ref{rec-ranks})}
is bounded. In this case $\sigma$ is given by
{\rm (\ref{sigma-r})}.
\end{lemma}

Now let $\alpha=(\alpha_k)$  be a strictly increasing sequence of positive integers. For
notational convenience we assume that the sequence is infinite, but our considerations also apply
to finite sequences with obvious modifications. Assume that
\begin{equation}\label{alpha-finite}
 \sum_{k=1}^{\infty} \frac{1}{\alpha_k}<\infty.
\end{equation}
Let $P^{(\alpha,1)}$ be the product measure on $[1]\times[2]\times\cdots$ which makes the coordinates
$r_i$  independent and satisfying
\begin{itemize}
\item[(i)]  $r_i\equiv i$  for $i\notin \{\alpha_1+1,\alpha_2+1,\dots\}$,
\item[(ii)] $r_i$ is uniformly distributed on $[i-1]$ for $i\in \{\alpha_1+1,\alpha_2+1,\dots\}$.
\end{itemize}

\begin{lemma}   A random sequence $(r_i)$ with distribution
$P^{(\alpha,1)}$  almost surely determines,  via {\rm (\ref{rec-ranks})} and {\rm
(\ref{sigma-r})}, a random element of $\Sym$.
\end{lemma}
\begin{proof}
Fix $i$ and condition on the event $r_i=s$. If $i\notin \{\alpha_1,\alpha_2,\dots\}$ then $r_{i,i+1}=s$.
If $i\in \{\alpha_1,\alpha_2,\dots\}$ then the expected value of $r_{i,i+1}$ is $s+s/i$. Iterating
we see that the expected value of $r_{i,j}$ converges, as $j\to\infty$,  to
$$ s\prod_{k\,:\, \alpha_k\geq i} \left(1+{1\over \alpha_k}\right),$$
which is finite in view of (\ref{alpha-finite}). Therefore, Fatou's lemma implies that $r_{i,j}$
is bounded in $j$ and Lemma \ref{inf-corr} can be applied.
\end{proof}
\noindent Using the correspondence between the  rank sequences $(r_i)$ and the infinite permutations we
consider $P^{(\alpha,1)}$ as a measure on $\O$ supported by $\Sym$. Similarly to $\Sym_n$, for
$\sigma\in\Sym$ we define position $j$ to be a record if $\sigma(j)=\max_{i\in[j]}\sigma(i)$. Then
under $P^{(\alpha,1)}$ the records are positions not of the kind $\alpha_k+1$.
Note that if the sequence $\alpha$ is finite, then $P^{(\alpha,1)}$ is supported by ${\mathfrak S}_\infty$.

\paragraph{The dual algorithm}
There is a dual stochastic algorithm that produces  $P^{(\alpha,1)}$  via  the entries of the
inverse infinite permutation $\sigma^{-1}(1), \sigma^{-1}(2),\cdots$. Note that the position of
integer $1$ belongs to  $\{1,\alpha_1+1,\alpha_2+1,\dots\}$, and that in terms of ranks we have
$\sigma^{-1}(1)=\max\{i: r_i=1\}$. Hence, introducing a random variable $\nu_1$ with distribution
\begin{eqnarray}
\label{rule-inverse1}
{\rm Prob}
\{\nu_1=0\}&=& \prod_{m=1}^\infty \left(1-{1\over \alpha_{m}}\right),  \\
\label{rule-inverse2} {\rm Prob} \{\nu_1=k\}&=&{1\over \alpha_k} \prod_{m=1}^\infty
\left(1-{1\over \alpha_{k+m}}\right), ~~~k=1,2,\dots,
\end{eqnarray}
the position of $1$ can be defined  as $\sigma^{-1}(1)=y_1$, where
$$y_1= 1(\nu_1=0)+ (\alpha_{\nu_1}+1) 1(\nu_1\neq 0).$$
Given the value $\sigma^{-1}(1)$,
define  a new sequence $\alpha'$ by the following rules
\begin{itemize}
\item[(i)] if $\sigma^{-1}(1)=1$ and $\alpha_1\geq 2$ then $\alpha_k'=\alpha_k-1$ for $k\geq 1,$
\item[(ii)] if $\sigma^{-1}(1)=1$ and $\alpha_1=1$ then $\alpha_k'=\alpha_{k+1}-1$ for $k\geq 1,$
\item[(iii)] if $\sigma^{-1}(1)=\alpha_i+1$ then  $\alpha'_k=\alpha_k$ for $k< i$ and
$\alpha_k'=\alpha_{k+1}-1$ for $k\geq i$.
\end{itemize}
Let $\nu_2$ be distributed as in the right-hand side of (\ref{rule-inverse1}),
(\ref{rule-inverse2}) but with $\alpha'$ in place of $\alpha$. Finally, let $\sigma^{-1}(2)$ be
the $y_2$th element of ${\mathbb N}\setminus \{\sigma^{-1}(1)\}$ for  $y_2= 1(\nu_2=0)+
(\alpha_{\nu_2}+1) 1(\nu_2\neq 0)$.  Then we iterate on ${\mathbb N}\setminus
\{\sigma^{-1}(1),\sigma^{-1}(2)\}$ and so on.

In Section
\ref{Section_laws_of_large_numbers} we will give a more direct proof that all positions
eventually get filled,  hence the output of the dual algorithm is indeed a random permutation $\sigma^{-1}\in\Sym$.

\vskip0.3cm

Finally, we construct a larger family of RD-measures by interpolating between  $P^*$ and $P^{(\alpha,1)}$'s.
Fix $\alpha$ satisfying (\ref{alpha-finite}) and $0<p\leq 1$. Split $\mathbb N$ in two infinite
subsets $N_1$ and $N_2$ by assigning each integer independently to $N_1$ with probability $p$ and
to $N_2$ with probability $1-p$. Using increasing bijections we can identify $N_1$ and $N_2$ with
two copies of $\mathbb N$. We construct an order by requiring that every $i\in N_1$ precedes every
$j\in N_2$ and (using the identifications with $\mathbb N$) by ordering $N_1$ according to $P^{(\alpha,1)}$ and
ordering $N_2$ according to   $P^*$. The distribution of the resulting order is denoted
$P^{(\alpha,p)}$.


Let $\Omega$ denote the set comprised of a point $*$ and of pairs $(\alpha,p)$, where $\alpha$ is
a strictly increasing sequence of positive integers
 satisfying
(\ref{alpha-finite}), and $0<p\le 1$.
The space $\Omega$ is a topological cone obtained by collapsing one face of the cylinder
$\{\alpha\}\times[0,1]$ in the point $*$. In this topology the convergence to $(\alpha,p)$ is
component-wise, and the convergence to $*$ means that the $p$-component goes to $0$. For  a generic
point $\omega$ of $\Omega$ (either $*$ or some $(\alpha,p)$), $P^\omega$ will denote the
corresponding measure.


\begin{theorem}
\label{theorem_extreme_points} Measures $P^{\omega}$ with $\omega\in \Omega$ comprise the set
${\rm ext}\,\mathcal M_R(\O)$ of the extreme RD-measures. The topology on $\Omega$ agrees with the
topology of weak convergence of measures on $\O$.
\end{theorem}
\noindent
Taken together with the uniqueness property of Choquet simplex, this result implies:



\begin{corollary}
\label{prop_simplex}
For every $P\in\mathcal M_R(\O)$
there
 exists a unique probability measure $\mu$ on $\Omega$, such that
$$  P=\int_{\Omega} P^{\omega} \mu (d\omega). $$
\end{corollary}

The next result is a law of large numbers for the extreme measures $P^\omega$.
\begin{proposition}
\label{prop_law_of_large_numbers}
Let $O$ be a random element
 of $\O$ distributed according to $P^{(\alpha,p)}$.
Then $P^{(\alpha,p)}$-almost surely
\begin{itemize}
\item[\rm(i)]
for every $k$ the $k$th
 non-record position in $\pi^{\infty}_n(O)$  converges to $\alpha_k+1$
as $n\to\infty$,

\item[\rm(ii)]
 $$
  \liminf_{n\to\infty} \frac{(\pi^{\infty}_n(O))^{-1}(n)}{n}=p.
 $$
\end{itemize}
 \end{proposition}

\smallskip
The proofs of Theorem \ref{theorem_extreme_points} and Proposition
\ref{prop_law_of_large_numbers} are postponed to later sections.

\noindent
 {\bf Remark 1. } If $\alpha_k$ is not defined, then under convergence to $\alpha_k +1 $ we mean that if
for large $n$ the $k$th non-record position in $\pi^{\infty}_n(O)$ exists, then it converges to
$+\infty$.


\noindent
 {\bf Remark 2. } Under $P^{(\alpha,p)}$ the number of records is asymptotically linear in $n$, so that
$|R(\pi^{\infty}_n(O))|/(np)\to 1$ in probability. Under $P^*$ the
 position of $n$ in $\pi^{\infty}_n(O)$ has uniform distribution on $[n]$, hence,
relation (ii) holds with $p=0$.



\section{The branching graph representation}
\label{Section_Branching_graph}

In this section we recast the setting of RD-measures on permutations within a general formalism of
central measures on branching graphs \cite{Kerov_book, KOO}.

The succession of permutations of different sizes and their record sets is representable
in the form of an infinite graded graph $\cal R$.
It is convenient to
 encode each admissible $\rho\subset [n]$  into
a binary word $\rho(1)\dots\rho(n)$ starting with $\rho(1)=1$. For instance,
$\{1,3,4\}\subset[5]$ becomes $10110$.
Let $\R_n$ denote the set of all $2^{n-1}$ such binary words of length $n$.

Consider  a  graded graph $\R$ with the set of vertices $\bigcup_{n=1}^{\infty} \R_n$, and with
edges connecting vertices on neighboring levels according to the rule: two vertices
$\rho=\rho(1)\dots \rho(n)\in\R_n$ and $\tau=\tau(1)\dots \tau({n+1})\in \R_{n+1}$ are connected
by an edge, denoted $\rho \nearrow \tau$, if there exists $k\in[ n+1]$ such that
\begin{itemize}
\item[(i)] $\rho(i)=\tau(i)$ for $i<k$,
\item[(ii)] $\tau(k)=1$,
\item[(iii)] $\tau(i)=0$ for $i>k$.
\end{itemize}
The first four levels of $\R$ are shown in Figure \ref{Fig_graph}.

\begin{figure}[h]
\begin{center}
\thicklines
\begin{picture}(300,200)

\put(5,100){$1$} \put(15,107){\line(1,1){40}} \put(15,101){\line(1,-1){40}}

\put(60,50){$10$} \put(75,51){\line(5,-1){40}} \put(75,54){\line(3,2){40}}
\put(75,57){\line(2,3){40}} \put(60,150){$11$} \put(75,157){\line(5,1){40}}
\put(75,154){\line(3,-2){40}} \put(75,150){\line(2,-3){40}}

\put(120, 40){$100$} \put(140,38){\line(5,-3){40}} \put(140,41){\line(6,0){40}}
\put(140,44){\line(2,1){40}} \put(140,47){\line(2,3){40}}

 \put(120, 80){$101$} \put(140,76){\line(2,-3){40}} \put(140,81){\line(3,-1){40}}
\put(140,84){\line(2,0){40}} \put(140,87){\line(3,2){40}}

 \put(120, 120){$110$} \put(140,119){\line(2,-5){40}} \put(140,122){\line(6,-1){40}}
\put(140,125){\line(3,1){40}} \put(140,128){\line(5,4){40}}

 \put(120, 160){$111$}  \put(138,158){\line(1,-3){45}} \put(140,161){\line(1,-1){40}}
\put(140,164){\line(3,0){40}} \put(140,167){\line(2,1){40}}

\put(180, 10){$1000$} \put(180, 35){$1001$} \put(180, 60){$1010$} \put(180, 85){$1011$} \put(180,
110){$1100$} \put(180, 135){$1101$} \put(180, 160){$1110$} \put(180, 185){$1111$}




\end{picture}
\end{center}
\caption{\label{Fig_graph} The first four levels of graph $\R$.}
\end{figure}
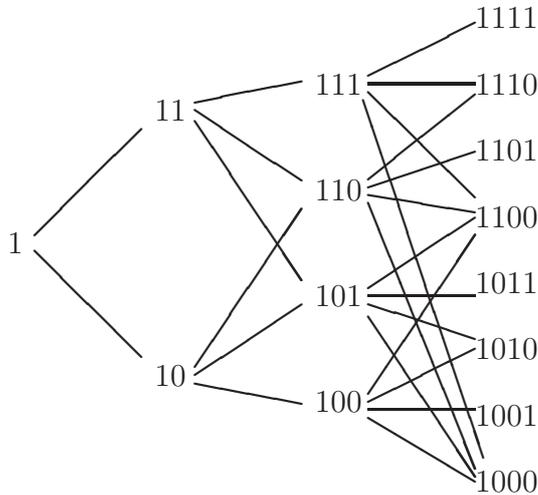

A (standard) path in $\R$ is a sequence of vertices $(\rho_i)$ such that $\rho_i\in \R_i$ and
$\rho_i\nearrow\rho_{i+1}$. Let  $\Gamma$ be the set of infinite paths
$\rho_1\nearrow\rho_2\nearrow\dots$, and let $\Gamma_n$ be the set of paths
$\rho_1\nearrow\ldots\nearrow\rho_n$ of length $n$. We view  $\Gamma$ as the projective limit of
finite sets $\Gamma_n$, and we equip $\Gamma$ with the usual topology of projective limit of
discrete spaces. Recall that $R(\sigma)\in\R_n$ for  $\sigma\in\Sym_n$.
\begin{proposition}
\label{prop_paths_are_permutations}
 The map
 $$  \Phi_n:  \sigma\to (R(\pi^n_j(\sigma)),~j\in [n])$$
(where $\pi_n^n$ is the identity map) is a bijection between $\Sym_n$ and $\Gamma_n$. Similarly,
the map
 $$  \Phi:\, O\to (R(\pi^\infty_n(O)), n\in{\mathbb N})$$
 is a homeomorphism between $\O$ and $\Gamma$.
\end{proposition}

\begin{proof} Let $\sigma_n=\pi^\infty_n(O)$.
As in the proof of Lemma \ref{prop_stability_of_record_dependence},
$\sigma_n^{-1}(n)$ is uniquely determined by
$R(\sigma_n)$ and $R(\sigma_{n-1})$. On the other hand,
$\sigma_n$ is uniquely determined by $\sigma_j^{-1}(j)$, $j\in [n]$, by the virtue of a correspondence analogous to
the bijection between $\Sym_n$ and the sequence of $n$ ranks.
\end{proof}

Identifying finite paths with permutations, and infinite paths with orders on $\mathbb N$ we use
the same symbols as above for measures and projections. For instance, $\pi_n^\infty$ denotes the
projection $\Gamma\to\Gamma_n$ which cuts the tail of a path up to the first $n$ terms.


A probability  measure $P_n$ on $\Gamma_n$ is called \emph{central} if the probability of a path
$\rho_1\nearrow \rho_2\dots\nearrow\rho_n,$ depends only on $\rho_n$, i.e.\ all paths with fixed
endpoint $\rho_n$ are equiprobable.
Similarly,  a probability measure $P$ on $\Gamma$ is central if all its projections
$P_n=\hat{\pi}_n^\infty(P)$ on $\Gamma_n$ are central. Let $\mathcal M_C(\Gamma)$ denote the space
of all central measures on $\Gamma$. We remark that each $P\in\mathcal M_C(\Gamma)$ is associated
with a random walk which moves along the paths in $\mathcal R$ and has standard (not depending on
$P$) backward transition probabilities determined by the condition of centrality, see
\cite{Kerov_book} for more details.

Now Proposition \ref{prop_paths_are_permutations} implies the following statement.
\begin{proposition}
\label{prop_record_dep_is_central}
 The map $\widehat \Phi$ is a an affine isomorphism of the convex sets
 $\mathcal M_R(\O)$ and $\mathcal M_C(\Gamma)$.
\end{proposition}

Let $\rho_n\in \R_n$. The  elementary measure $P^{\rho_n}$ is a unique central measure on
$\Gamma_n$ supported by the set of paths of length $n$ with endpoint $\rho_n$. Under the bijection
between $\Sym_n$ and $\Gamma_n$, the elementary measure corresponds to the uniform distribution on
the set of permutations $\sigma\in\Sym_n$ with fixed records $R(\sigma)=\rho_n$. The next standard
fact (see e.g.\ \cite[Proposition 10.8]{Olsh} or
\cite[Theorem 1.1]{DF}) is the main technical tool to identify the set
$\mathcal M_C(\Gamma)$.

\begin{lemma}
\label{theorem_approximation}
 Let $P$ be an extreme point of the convex set $\mathcal M_C(\Gamma)$.
 Then for $P$-almost all paths
 $\rho_1\nearrow\rho_2\nearrow\dots$ in  $\cal R$
 \begin{equation}
 \label{eq_approximation_theorem}
  \lim_{n\to\infty}\hat{\pi}_k^n( P^{\rho_n} )(A)= \hat{\pi}_k^\infty(P)(A),
 \end{equation}
for all $A\subset \Gamma_k$ and $k\in {\mathbb N}$.
\end{lemma}

The family of probability measures $P\in {\mathcal M}_C({\mathcal R})$ representable as limits
(\ref{eq_approximation_theorem}) of elementary measures along some paths $(\rho_n)$ is
called the  \emph{ Martin boundary} of the graph $\R$. (We remark that sometimes the Martin
boundary is defined as a larger  set of limits along arbitrary sequences
$(\rho_n)$.) By Lemma \ref{theorem_approximation} the set of extremes ${\rm ext}\,\mathcal
M_C(\Gamma)$ is a part of the Martin boundary. Convergence (\ref{eq_approximation_theorem}) is the
same as the weak convergence of projections on every $\Gamma_k$. With this  in mind we simply
write  $P^{\rho_n}\to P$. The boundary problem has a straightforward  reformulation in terms of
permutations.


\section{The Martin boundary identification}
\label{Section_Martin_boundary}


\begin{theorem}
\label{theorem_Martin_boundary}
 Let $(\rho_n)\in\Gamma$ be a path such that
the elementary RD-measures on $\Sym_{n}$ weakly converge,  i.e.\ $P^{\rho_n}\to P.$ Then
$P=P^\omega$ for some  $\omega\in\Omega$. Thus, the Martin boundary of $\mathcal R$ can be
identified with
 $\Omega$.
\end{theorem}
\noindent
To prove the result we need a number of auxiliary propositions.

 For  $\rho\in\R_n$ with $k$ zeros let  $\ell_1(\rho)<\ell_2(\rho)<\dots<\ell_k(\rho)$ be the positions of zeros
 listed in increasing order,
 $$
  \{\ell_1(\rho),\dots,\ell_k(\rho)\}=\{i:\rho(i)=0\},
 $$
and define
 $$
  L(\rho):=\prod_{i=1}^k\left(1-\frac{1}{\ell_i(\rho)-1}\right).
 $$

The following algorithm produces a random permutation with distribution $P^{\rho}$. Let
$m_1>m_2>\dots>m_{n-k}=1$ be the positions of 1's listed in decreasing order. Since
$\sigma^{-1}(n)$ is the largest record we must have $\sigma(m_1)=n$. Next,
$\sigma(m_1+1),\dots,\sigma(n)$ is an equiprobable sample without replacement from
$\{1,\dots,n-1\}=\{1,\dots,n\}\setminus \{\sigma(m_1)\}$. Furthermore, $\sigma(m_2)$ is the
maximal element of $\{1,\dots,n\}\setminus\{\sigma(m_1),\dots,\sigma(n)\}$, thus there is only one
choice for $\sigma(m_2)$ after $\sigma(m_1),\dots,\sigma(n)$ have been determined. Now
$\sigma(m_3+1),\dots,\sigma(m_2-1)$ is an equiprobable sample from
$\{1,\dots,n\}\setminus\{\sigma(m_1),\dots,\sigma(n)\}\setminus\{\sigma(m_2)\}$. The process is
continued until all positions are filled.

The above algorithm for sampling permutations from  $P^{\rho}$ readily implies the following.
\begin{proposition}
\label{prop_pos_1_finite} Let $\rho\in\R_n$ and let $\sigma$ be a random permutation from $\Sym_n$
distributed according to $ P^{\rho}$. The random variable $\sigma^{-1}(1)$ has the
following distribution:
$$
  P^{\rho}(\sigma^{-1}(1)=h)=\begin{cases}
  \prod_{i=1}^k \left(1-\frac{1}{\ell_i-1}\right),\quad{\rm if ~~} h=1,\\
  \frac{1}{\ell_j-1} \prod_{i=j+1}^k \left(1-\frac{1}{\ell_i-1}\right),\quad{\rm if}~~ h= \ell_j,\\
  0,\quad\text{\rm otherwise.}
 \end{cases}
$$
\end{proposition}

\begin{proposition}
\label{prop_pos234_finite} Let $\rho\in\R_n$, and let $\sigma$ be a random permutation from $\Sym_n$
distributed according to $ P^{\rho}$.  For $2\le t\le n$
the conditional distribution of
 $\sigma^{-1}(t)$ given
$\sigma^{-1}(1)=s_1$, \dots, $\sigma^{-1}(t-1)=s_{t-1}$  is
\begin{multline*}
 { P}^{\rho}(\sigma^{-1}(t)=h\mid \sigma^{-1}(1)=s_1, \dots,
\sigma^{-1}(t-1)=s_{t-1} )=\\
\begin{cases}
\prod_{i=1}^{k'} \left(1-\frac{1}{\ell'_i-1-w(\ell'_i)}\right)   ,
\quad{\rm if ~~} h=\min(\{1,\dots,n\}\setminus\{s_1,\dots,s_{t-1}\}),\\
\frac{1}{\ell'_j-1-w(\ell'_j)} \prod_{i=j+1}^{k'}
\left(1-\frac{1}{\ell'_i-1-w(\ell'_j)}\right),\quad{\rm if~~} h= \ell'_j,\\
0,\quad\text{ {\rm otherwise,}}
\end{cases}
\end{multline*}
where  $\ell'_1<\dots<\ell'_{k'}$ satisfy
$$
 \{\ell'_1,\dots,\ell'_{k'}\}=\{\ell_1,\dots,\ell_k\}\setminus\{s_1,\dots,s_{t-1}\}
$$
and
$$
w(x)=\left|\{s_1,\dots,s_{t-1}\}\cap \{1,\dots,x-1\}\right|.
$$
\end{proposition}

\begin{proposition}
\label{prop_conditions_for_convergence_to_uniform}
 If a  path $(\rho_n)\in \Gamma$ satisfies $L(\rho_n)\to 0$, then
$P^{\rho_n}\to P^*.$
\end{proposition}
\begin{proof} Observe that
if for $\sigma\in\Sym_n$ positions of $1,\dots,m$ are not records, i.e.\ if
$\{\sigma^{-1}(1),\dots,\sigma^{-1}(m)\} \subset[n]\setminus R(\sigma)$, then the set of records
remains unaltered when the positions of $1,\dots,m$ are exchanged. Therefore, under the RD-measure
$P^{\rho_n}$  the permutation
 $\pi^n_m(\sigma)$ is uniformly distributed given
 $\{\sigma^{-1}(1),\dots,\sigma^{-1}(m)\}\subset [n]\setminus R(\sigma)$.
Finally,
by Propositions \ref{prop_pos_1_finite} and \ref{prop_pos234_finite}
 if $L(\rho_n)\to 0$, then
$$
 P^{\rho_n}(\sigma^{-1}(m)\notin R(\sigma))\to 1,
$$
hence $\hat{\pi}_m^n(P^{\rho_n})$ converges to the uniform distribution on ${\mathfrak S}_m$,
for every $m$.
\end{proof}

\begin{proposition}
\label{prop_convergence_to_non_uniform} Let $(\rho_n) \in\Gamma$ be a path such that
$P^{\rho_n}\to P$ and $P\neq P^*$. Then there exists a $0\!-\!1$
 sequence $\rho_{\infty}=(\rho_\infty(1),\rho_{\infty}(2),\dots)$ such that
 $$
  \lim_{n\to\infty} \rho_n(i)=\rho_\infty(i)
 $$
 for every $i$.
\end{proposition}

\begin{proof}
Suppose that for some $i$ the sequence $\rho_n(i)$ does not converge. Then for infinitely
many $n_j$ we have $\rho_{n_j}(i)=0$ and $\rho_{n_j-1}(i)=1$. Then, since
$\rho_{n_j-1}\nearrow \rho_{n_j}$, we have
$$\rho_{n_j}(i)=\rho_{n_j}({i+1})=\dots=\rho_{n_j}({n_j})=0.
$$
Therefore, $L(\rho_{n_j})\to 0$ as $j\to\infty$ and Proposition
\ref{prop_conditions_for_convergence_to_uniform} implies that
$\widehat \pi^{n_j}_k( P^{\rho_{n_j}})$ converges to the uniform measure on $\Sym_k$,
so $P^{\rho_{n_j}}\to P^*$ which is a contradiction.
\end{proof}

Now we are ready to prove Theorem \ref{theorem_Martin_boundary}.
\begin{proof}[Proof of Theorem \ref{theorem_Martin_boundary}]
If $P^{\rho_n}\to P^*$ then there is nothing to prove. Otherwise by Propositions
\ref{prop_conditions_for_convergence_to_uniform} and \ref{prop_convergence_to_non_uniform}, passing if necessary to a subsequence, we have as $n\to\infty$
 \begin{enumerate}
 \item $L(\rho_n)\to p_1$ for some $0<p_1\leq 1$,
 \item $\lim_{n\to\infty} \rho_n(i)=\rho_\infty(i)$ for some
$0-1$ sequence $\rho_\infty=(\rho_\infty(1),\rho_\infty(2),\dots)$.
 \end{enumerate}

Let $\ell_1(\rho_{\infty})<\ell_2(\rho_{\infty})<\dots$ be positions of zeros
 in $\rho_{\infty}$:
 $$
  \{\ell_1(\rho_\infty),\ell_2(\rho_{\infty}),\dots\}=\{i:\rho_{\infty}(i)=0\}.
 $$
 Set $\alpha_i=\ell_i(\rho_\infty)-1$ for all $i$ such that $\ell_i(\rho_\infty)$ is defined.
Observe that convergence of $L(\rho_n)$ entails that
 $\prod_i(1-1/\alpha_i)$ converges to some $p_2$  with $p_1\le p_2\le 1$. Now set $p=p_1/p_2$. We
 claim that
 $$
 P_k=\widehat \pi^\infty_k(P^{(\alpha,p)}).
 $$
 The claim is shown by comparing the description of $P^{(\alpha,p)}$
 via the dual algorithm given in Section \ref{Section_Statement}) with  the description of
the elementary measures $P^\rho$ given in Propositions \ref{prop_pos_1_finite} and
 \ref{prop_pos234_finite}.
\end{proof}

\begin{corollary}
\label{corollary_record_dependent}
 The measures $P^\omega$, $\omega\in\Omega$, are record-dependent.
\end{corollary}
\begin{proof}
Indeed, by Theorem \ref{theorem_Martin_boundary} they are weak limits of record-dependent
measures.
\end{proof}

\section{The laws of large numbers}
\label{Section_laws_of_large_numbers}

In this section we exploit the algorithmic description of measures $P^{(\alpha,p)}$ to prove
Proposition \ref{prop_law_of_large_numbers} and to finish the proof of Theorem
\ref{theorem_extreme_points}.

First, suppose that $p=1$ and fix a sequence $\alpha$ such that $\sum_{i=1}^\infty 1/\alpha_i<\infty$.
Recall, that the dual algorithm for $P^{(\alpha,1)}$ constructs successively the entries
$\sigma^{-1}(1),\sigma^{-1}(2),\dots$
of the inverse permutation
$\sigma^{-1}:\mathbb N\to\mathbb N$.
\begin{lemma}
\label{lemma_estimate_in_algorithm}
 For every $\varepsilon>0$ there exist constants $C>1$ and $n_0$
such that the estimate
 \begin{equation}\label{condpro}
  {P^{(\alpha,1)}}(\sigma^{-1}(k)>Cn\mid \sigma^{-1}(1)=s_1,\dots,\sigma^{-1}(k-1)=s_{k-1})<\varepsilon.
\end{equation}
holds for $n>n_0$, $k\leq n$ and arbitrary distinct  $s_1,\dots,s_{k-1}$.

\end{lemma}
\begin{proof}
For shorthand, we write $Q$ for the conditional probability in (\ref{condpro}).
 As follows from the description of the dual algorithm in Section  \ref{Section_Statement},
 $$
  {P^{(\alpha,1)}}(\sigma^{-1}(1)>Cn)=1-\prod_{i\,:\alpha_i>Cn}\left(1-\frac1{\alpha_i}\right).
 $$
 More generally, a similar formula holds for $Q$
 with $\alpha_i$ being replaced by  other sequence $\beta_i$.
Following a procedure  in Section \ref{Section_Statement} to derive $(\beta_i)$,
we pass from $(\alpha_i)$ to a subsequence and
 then subtract  from each  term a nonnegative integer not exceeding $k$.
 Therefore,
\begin{equation*}
Q \le   1-\prod_{i:\,\alpha_i>Cn}\left(1-\frac1{\alpha_i-k}\right)  \le
1-\prod_{i:\,\alpha_i>Cn}\left(1-\frac1{\alpha_i-n}\right) .
 \end{equation*}

Since $\ln(1+x)\ge 2x$ for $-1/2\le x\le 0$, we have the following estimate
\begin{multline*}
 -\frac{1}{2}\ln\left(\prod_{i:\,\alpha_i>Cn}\left(1-\frac1{\alpha_i-n}\right)\right) \le
 \sum_{i:\,\alpha_i>Cn}\frac1{\alpha_i-n}\\=  \sum_{i:\,\alpha_i>Cn}\frac1{\alpha_i}+
 \sum_{i:\,\alpha_i>Cn}\frac{n}{\alpha_i(\alpha_i-n)} \le
 \sum_{i:\,\alpha_i>Cn}\frac1{\alpha_i}+
 \sum_{j=Cn}^{\infty}\frac{n}{(j-n-1)(j-n)}\\=  \sum_{i:\,\alpha_i>Cn}\frac1{\alpha_i}+
 \frac{n}{Cn}= \sum_{i:\,\alpha_i>Cn}\frac1{\alpha_i}+
 \frac{1}{C}
\end{multline*}

Now choose small enough $\delta>0$ to have  $1-e^{-\delta}<\varepsilon$. Let $C>\frac{1}{4\delta}$ and
choose $n_0$ such that
$$
 \sum_{i:\,\alpha_i>C{n_0}}\frac1{\alpha_i} <\delta/4
$$
(this is possible, since $\sum 1/\alpha_i$ converges). Then for $n>n_0$ we obtain
$Q<1-e^{-\delta}<\varepsilon$, as desired.
\end{proof}

\begin{proposition}
 \label{prop_large_numbers_unmixed}
Let $O$ be a random order with distribution $P^{(\alpha,1)}$ and let $\sigma_n=\pi^\infty_n(O)$ be
the projection of  $O$ on $\Sym_n$. Then $P^{(\alpha,1)}$-almost surely
$$
  \frac{\sigma^{-1}_n(n)}{n} \to 1.
 $$
\end{proposition}
\begin{proof}
 Choose $\varepsilon>0$. Recall that a real-valued random variable $X$ stochastically dominates
another such
variable $Y$ if for any bounded non-decreasing function $f$ the expected values satisfy  $Ef(X)\ge Ef(Y)$. Observe that by
Lemma \ref{lemma_estimate_in_algorithm}  the random variable $X=|\{1\le i\le n\mid
\sigma^{-1}(i)\le Cn \}|$ stochastically dominates a sum of $n-n_0$ independent Bernoulli random
variables with the probability of $1$ equal to $1-\varepsilon$
(see Lemma 1.1 in and \cite{LSS}   and Lemma 1 in
\cite{R}). Now using a standard large deviations estimate
for the sum of independent Bernoulli random variables (see e.g. \cite{K}, Chapter 27),
we conclude that there exist constants
$C_1>0$ and $C_2>0$ such that
 \begin{equation}
 \label{eq_x1}
  {P^{(\alpha,1)}}(\left|\{1\le i\le n\mid \sigma^{-1}(i)\le Cn\}\right|>
  (1-2\varepsilon)n) > 1 - \exp(-C_2 n)
 \end{equation}
 for $n>n_1$.


Observe that
the set $\{\sigma^{-1}(1),\dots,\sigma^{-1}(n)\}$ is
the union of an integer interval $\{1,\dots,M\}$ and a subset of the set
$\{\alpha_1+1,\alpha_2+1,\dots\}$.

The convergence of the series $\sum 1/\alpha_i$ implies that
$$
 \frac{\left| \{i\mid \alpha_i \le Cn\}\right|}{n}\to 0.
$$
Therefore, \eqref{eq_x1} implies that
$$
  {P^{(\alpha,1)}}\left(\{1,\dots,\lfloor (1-3\varepsilon )n\rfloor\}\subset \{ \sigma^{-1}(1),\dots,\sigma^{-1}(n)\}\right) > 1-\exp(-C_2 n)
$$
for $n>n_2$. But on the event
$\{1,\dots,\lfloor (1-3\varepsilon) n\rfloor\}\subset \{
\sigma^{-1}(1),\dots,\sigma^{-1}(n)\}$ we have $(\pi^{\infty}_{n+1}(O))^{-1}(n+1)>
(1-3\varepsilon) n$. Hence, for $n>n_2$ we have
\begin{equation}
\label{eq_x2}
 {P^{(\alpha,1)}}\left(\frac{\sigma^{-1}_{n+1}(n+1)}{n}> (1-3\varepsilon)\right)>1-\exp(-C_2 n).
\end{equation}
Since
$
\sum_{n=n_2+1}^{\infty} \exp(-C_2 n)<\infty,
$
from \eqref{eq_x2} and the Borel-Cantelli lemma follows that  for all but finitely many $n$ we have
$$
 \frac{\sigma^{-1}_{n+1}(n+1)}{n}> 1-3\varepsilon,
$$
whence
$$
 \liminf_{n\to\infty} \frac{\sigma^{-1}_{n}(n)}{n}>1-3\varepsilon
$$
almost surely.
To finish the proof it remains to observe that $\varepsilon>0$ is arbitrary and
$\sigma^{-1}_{n}(n)\le n$ always holds.
\end{proof}

\begin{corollary}
\label{corollary_algorithm_gives_a_bijection}
The dual algorithm for $P^{(\alpha,1)}$ eventually fills every position,
so that
the output  is indeed a
bijection $\sigma: \mathbb N\to \mathbb N$.
\end{corollary}
\begin{proof}
 Indeed, in the proof of Proposition \ref{prop_large_numbers_unmixed} we have shown that for every
 $k$ the probability of the event $\{1,\dots,k\}\subset\{\sigma^{-1}(1),\dots,\sigma^{-1}(n)\}$
 tends to $1$ as $n\to\infty$.
\end{proof}

Now we seek for an analogue of Proposition \ref{prop_large_numbers_unmixed} for more general
$P^{\omega}$.

\begin{proposition}
\label{prop_large_numbers_for_zero} Let $O$ be a random order with distribution $P^{*}$ and let
$\sigma_n=\pi^\infty_n(O)$ be the projection of  $O$ on $\Sym_n$. Then $P^{*}$-almost surely
 $$
 \liminf_{n\to\infty} \frac{\sigma_n^{-1}(n)}{n} \to 0.
 $$
\end{proposition}
\begin{proof}
Under $P^*$ the permutation $\sigma_{n-1}$ and the position $\sigma^{-1}_n(n)$ are independent,
and the latter is uniformly distributed on $[n]$. Since $P^*(\sigma_n^{-1}(n)=1)=1/n$, the event
$\{\sigma_n^{-1}(n)=1\}$ almost surely occurs infinitely often as $n\to\infty$, and the statement
becomes trivial.
\end{proof}


The analogous statement for general $P^{(\alpha,p)}$  interpolates between Propositions
\ref{prop_large_numbers_unmixed} and \ref{prop_large_numbers_for_zero}.
\begin{proposition}
 Let $0<p<1$, then for the order $O$ with distribution $P^{(\alpha,p)}$, almost surely
 $$
  \liminf_{n\to\infty} \frac{(\pi^{\infty}_n(O))^{-1}(n)}{n} \to p.
 $$
\end{proposition}
\begin{proof}
Let $O_0$ and $O_\alpha$ be two independent linear orders on $\mathbb N$, such that the
distributions of $O_0$ and $O_\alpha$ are $P^*$ and $P^{(\alpha,1)}$, respectively. Recall that
the $P^{(\alpha,p)}$-distributed order $O$ is constructed from $O_0$ and $O_\alpha$ by splitting $\mathbb N$ into two subsets $ N_1$ and $N_2$
(with the aid of a coin landing heads up with probability $p$), setting $O_0$
on $ N_2$ and  $O_\alpha$ on $ N_1$,  and requiring that
 $ N_1$ precedes $ N_2$.

Projecting to $[n]$ yields $\sigma=\pi^\infty_n(O)$, constructed as follows. Let $M_1= N_1\cap [n]$
and $M_2= N_2 \cap [n]$. The permutation $\sigma_1=\pi^{\infty}_{|M_1|}(O_\alpha)$
uniquely defines a permutation $\bar \sigma_1$
of the set $M_1$ and $\sigma_2=\pi^{\infty}_{|M_2|}(O_0)$ uniquely defines a
permutation $\bar \sigma_2$
of the set $M_2$. Permutation $\sigma$ is obtained
by first writing $\bar \sigma_1$ and then writing $\bar\sigma_2$.

Let us analyze $\sigma^{-1}(n)$. Choose $\varepsilon> 0$. Almost surely for large enough $n$ we
have
\begin{enumerate}
\item
 $p-\varepsilon\le |M_1|/n\le p+\varepsilon$,
\item
 $1-\varepsilon\le (\sigma_1)^{-1}(|M_1|)/|M_1|\le 1$.
\end{enumerate}
 The latter is just the statement of Proposition \ref{prop_large_numbers_unmixed} and the former
 follows from the law of large numbers for Bernoulli trials. Now if $n\in M_1$, then
$\sigma^{-1}(n)=\sigma_1^{-1}(|M_1|)$ and, thus,
$$
  \frac{\sigma^{-1}(n)}{n}\ge(p-\varepsilon)(1-\varepsilon).
$$
If $n\in M_1$, then
$$
\sigma^{-1}(n)>|M_1|\ge (p-\varepsilon)n.
$$
Since $\varepsilon$ is arbitrary, we conclude that
$$
\liminf_{n\to\infty} \frac{(\pi^\infty_n(O))^{-1}(n)}{n} \ge p.
$$

Next, using Proposition \ref{prop_large_numbers_for_zero} we conclude that almost surely there
exists an increasing sequence $n_m$ such that for $n=n_m$, $m=1,2,\dots$ we have
\begin{enumerate}
 \item
  $n\in M_2$,
 \item
 $\sigma_2^{-1}(|M_2|)=1.$
\end{enumerate}
This implies that for large enough $m$,
$$
\sigma^{-1}(n)=|M_1|+1\le(p+\varepsilon)n+1.
$$
Therefore,
$$
\liminf_{n\to\infty} \frac{(\pi^\infty_n(O))^{-1}(n)}{n} \le (p+\varepsilon)
$$
Since $\varepsilon$ is arbitrary, we are done.
\end{proof}

\begin{proposition}
 If $p>0$, then under $P^{(\alpha,p)}$ the position of the $i$th non-record in
$\pi^{\infty}_k(O)$ converges to $\alpha_i+1$ as $k\to\infty$
almost surely.
\end{proposition}
\begin{proof}
 First, suppose that $p=1$ and
recall the algorithmic description of $P^{(\alpha,1)}$. The permutation
$\pi^{\infty}_k(O)$ is read from the order of numbers $1,\dots,k$ after the first $k$ steps of the
algorithm. Moreover, observe that if after $k$ steps of the algorithm all positions
$1,\dots,\alpha_i+1$ are filled, then $\alpha_i+1$ is precisely the position of the
$i$th non-record in $\pi^{\infty}_k(O)$. Therefore, our claim is implied by  Corollary
\ref{corollary_algorithm_gives_a_bijection}.

For the general $p$ a bulk of integers is appended  at the right end of the permutation, thus not affecting  positions of the first few  non-records.
\end{proof}


\begin{proof}[Proof of Theorem \ref{theorem_extreme_points}]
The set of extremes ${\rm ext}\,{\cal M}_R({\cal O})$ is contained in the Martin boundary by Lemma
\ref{theorem_approximation}. On the other hand, by Proposition \ref{prop_law_of_large_numbers}
each measure $P^\omega, \omega\in\Omega$, satisfies a law of large numbers specific for
this particular $P^\omega$. It follows that the supports of $P^\omega$'s are disjoint, hence none
of the measures
can be represented as a nontrivial convex mixture  over the Martin boundary. Thus every
$P^\omega$ is extreme, so ${\rm ext}\,{\cal M}_R({\cal O})=\{P^\omega,~\omega\in\Omega\}$. The
coincidence of topologies immediately follows from the explicit description of measures $P^\omega$
given in Section \ref{Section_Statement}.
\end{proof}


\section{Two connections}

\label{Section_connections}


\paragraph{Order-invariant measures on causal sets}
We describe now a connection of the record-dependent measures $P^{(\alpha,1)}$
to a recent work on random partial orders \cite{BL1,BL2}.

A partial order $\triangleleft$ on $\mathbb N$ defines a
{\it causal set}  $({\mathbb N},\triangleleft)$ if every element is preceded by finitely many other elements.
A {\it natural extension} of $\triangleleft$ is an order-preserving bijection
$\sigma:{\mathbb N}\to {\mathbb N}$, i.e.  $i\triangleleft j$ implies
$\sigma^{-1}(i)<\sigma^{-1}(j)$. A {\it stem} is a finite collection of positions $j_1,\dots,j_k$ such that there
exists a natural extension with $\sigma^{-1}(1)=j_1,\dots,\sigma^{-1}(k)=j_k$.

If  $j_1,\ldots,j_k$ is a stem, every $D_\ell=\{j_1,\dots,j_\ell\}$, $1\leq \ell\leq k$, is a down-set  (lower ideal).
A stem can be identified with a chain of down-sets $D_1\subset\cdots\subset D_k$, where $|D_\ell|=\ell$.
It is not hard to see that an
infinite chain of down-sets  $D_1\subset D_2\subset\ldots$ (where $|D_\ell|=\ell$) with $\cup D_\ell={\mathbb N}$
uniquely corresponds to a natural extension of $\triangleleft$.

Brightwell and Luczak \cite {BL1,BL2} defined an {\it order-invariant measure}
as a probability measure $P$ on the set of natural extensions of $\triangleleft$, such that
$$P(\sigma^{-1}(1)=j_1,\dots,\sigma^{-1}(k)=j_k)=P(\sigma^{-1}(1)=\ell_1,\dots,\sigma^{-1}(k)=\ell_k),$$
provided $\{j_1,\dots,j_k\}=\{\ell_1,\dots,\ell_k\}$. The condition means that the probability of a stem only depends on
the corresponding down-set $D_k=\{j_1,\dots,j_k\}$.
It is possible to interpret order-invariant measures as  central measures  on the path
space of a graded graph of down-sets.

Let $(\alpha_k)$ be a strictly increasing sequence of integers as in Section 3,
and let $(\beta_k)$ be the (infinite) sequence complimentary to $(\alpha_k+1)$, so that
$\{\alpha_1+1,\alpha_2+1,\ldots\}\cup \{\beta_1,\beta_2,\dots\}={\mathbb N}$.
Consider a partial order $\triangleleft$  generated by the relations
$$\beta_1\triangleleft\beta_2\triangleleft\ldots, \qquad  \alpha_i+1\,\triangleleft\,
\max\{\beta_k: \beta_k\leq \alpha_i\},$$
which mean that $(\beta_k)$ is a chain, and each segment $\beta_k+1,\beta_k+2,\dots,\beta_{k+1}-1$ is an antichain
covered by $\beta_k$.

Obviously from the definitions, $\sigma:{\mathbb N}\to {\mathbb N}$ is a natural extension of $\triangleleft$ if and
only if $\{\beta_k\}$ is the set of records of $\sigma$.

\begin{proposition} $P^{(\alpha,1)}$ is a unique order-invariant measure for the causal set
$({\mathbb N},\triangleleft)$.
\end{proposition}

\begin{proof}[Sketch of the proof] Every finite down-set with elements arranged in increasing order is a
 sequence $\gamma$ of the kind
$$1,2,\dots,\beta_k,\beta_k+1,\beta_k+2,\dots,\beta_{k+1}-1,\alpha_{i_1}+1,\dots,\alpha_{i_\ell}+1,$$
where either of the segments $\beta_k+1,\beta_k+2,\dots,\beta_{k+1}-1$ or
$\alpha_{i_1},\dots,\alpha_{i_\ell}$ can be empty.  Call such $\gamma$ admissible.

The conditions  ensuring that the set of records is $(\beta_k)$ and that $\gamma$ is admissible
impose constraints on permutation that can be expressed in terms of the  $r_i$. We illustrate this
with $\gamma$ of the form
$$1,\dots,a,b,c$$
where $a+1\in \{b_k\}$ and  $b,c\in \{\alpha_k+1\}$. The constraints on the ranks become
$r_j=j$ for $j\notin \{\alpha_k+1\}$, $r_j<j$ for  $j\in \{\alpha_k+1\}$ and, to guarantee the admissibility,
\begin{eqnarray*}
r_i&\geq& a+1 {\rm ~~for~~} a<i<b,\\
r_b&\leq& a+1,\\
r_i&\geq & a+2 {\rm ~~for~~} b<i<c,\\
r_c&=&a+2,\\
r_i&\geq& a+3 {\rm ~~for~~} i>c.
\end{eqnarray*}

Under $P^{(\alpha, 1)}$ the $r_i$'s are independent and
each $r_i$ is
uniformly distributed on a suitable range.
Therefore each possible
stem associated with $\gamma$  has the same probability, equal to the probability of admissible realization of
  $r_j, j\leq \alpha_{i_\ell}$. The order-invariance of the measure follows.

For  a {\it finite} causal set $([n],\triangleleft)$ the analog of order-invariant measure is the
uniform distribution on the extensions of $\triangleleft$. The uniqueness assertion follows from
the fact that $P^{(\alpha,1)}$ is a weak limit of such measures as $n\to\infty$ along $(\beta_i)$,
and condition (\ref{alpha-finite}) ensures that the limit is a bijection. We omit details,  see
\cite[Section 9]{BL2} for a more general result.
\end{proof}

\paragraph{The Young-Fibonacci lattice}

The \emph{Young-Fibonacci graph} (lattice) was introduced by Stanley \cite{YF_St} and Fomin
\cite{YF_F}. They found out that it shares lots of the features with the Young graph, which is the
object naturally arising in the theory of group representations and combinatorics. In particular, Stanley proved
that both graphs are \emph{differential posets}.

 The vertices of the Young-Fibonacci graph at level $n$ are labeled
 by words  in the alphabet $\{1,2\}$, with the sum of  digits equal $n$.
For instance 1111, 211, 121, 112, 22 are all the words on level $n=4$. The number of vertices on $n$th
level is the $n$th Fibonacci number. Successors of a word are obtained by either inserting a 1 in
any position within the leftmost contiguous block of 2's, or by replacing the leftmost 1 with 2.
For instance, $2212$ has successors $12212, 21212, 22112, 2222$.

Goodman and Kerov \cite{YF1} studied the Martin boundary of the Young-Fibonacci graph.
 Comparing with their result, it is seen that the Martin boundary of the Young-Fibonacci graph
 has the same conical structure as  our $\Omega$. The apex is the {\it Plancherel measure},
 which (like our $P^*$) appears as a pushforward of the uniform
 distribution on permutations. The base is a discrete space comprised of the measures which (like
our $P^{(\alpha,1)}$'s) are parametrized by infinite words in the alphabet $\{1,2\}$ with `rare'
occurrences of 2's, to satisfy  a  condition similar to \eqref{alpha-finite}. The Plancherel
measure of the Young-Fibonacci graph was further studied in \cite{YF2}.

The arguments of \cite{YF1} are very much different from the present paper. Goodman and Kerov
intensively use the relation to a certain non-commutative algebra introduced by Okada \cite{Ok}.
Note also that unlike the Young-Fibonacci graph, the graph of record-sets $\cal R$ is not a
differential poset. Thus, it seems that no direct connection of central measures on  $\cal R$ and the
Young-Fibonacci graph exist. This makes the coincidence of the Martin boundaries
even more intriguing.


{\bf Acknowledgements.} V.G.\, was partially supported by RFBR-CNRS grants 10-01-93114 and
11-01-93105.

\end{document}